\documentclass[12pt]{article}
\usepackage{amsmath,enumerate,amsfonts,amssymb,color,graphicx,amsthm}
\usepackage{multirow,caption,subcaption}
\usepackage{hyperref}
\usepackage{todonotes}
\usepackage[normalem]{ulem}
\usepackage{cite}

\numberwithin{equation}{section}

\def\R{{\mathbb R}}

\def\Om{{\Omega}}
\def\lam{{\lambda}}
\def\uk{{\overline{u_k}}}
\def\rh{{\overline{\rho}}}
\def\al{{\overline{\alpha}}}
\def\ep{{\epsilon}}
\def\ga{{\gamma_{\epsilon}}}
\def\p{{\partial}}
\def\no{{\nonumber}}
\newcounter{marnote}

\begin{document}
\newtheorem{lem}{Lemma}[section]
\newtheorem{rem}{Remark}
\newtheorem{question}{Question}
\newtheorem{prop}{Proposition}
\newtheorem{cor}{Corollary}
\newtheorem{thm}{Theorem}[section]
\newtheorem{definition}{Definition}
\newtheorem{openproblem}{Open Problem}
\newtheorem{conjecture}{Conjecture}

\newenvironment{dedication}
  {
   \vspace*{\stretch{.2}}
   \itshape             
  }
  {
   \vspace{\stretch{1}} 
  }

\title
{A singular Kazdan-Warner problem on a compact Riemann surface}
\author{
Xiaobao Zhu\thanks{Partially
 supported by the National Science Foundation of China
        (Grant Nos.
11171347 and 11401575).
}
\thanks{Email:  zhuxiaobao@ruc.edu.cn}
\\
School of Mathematics\\
Renmin University of China\\
Beijing 100872, P. R. China\\
}
\date{ }
\maketitle


\begin{abstract}
Let $(M,g)$ be a compact Riemann surface with unit area,
 $h\in C^{\infty}(M)$ a function which is positive somewhere, $\rho>0$, $p_i\in M$ and $\alpha_i\in(-1,+\infty)$ for $i=1,\cdots,\ell$, we consider the mean field equation
\begin{align*}
\Delta v + 4\pi\sum_{i=1}^{\ell}\alpha_i\left(1-\delta_{p_i}\right) = \rho\left(1-\frac{he^v}{\int_Mhe^vd\mu}\right),
\end{align*}
on $M$, where $\Delta$ and $d\mu$ are the Laplace-Beltrami operator and the area element of $(M,g)$ respectively. Using variational method and blowup analysis, we
prove some existence results in the critical case $\rho=8\pi(1+\min\{0,\min_{1\leq i\leq\ell}\alpha_i\})$. These results can be seen as partial generalizations of works of
Chen-Li (J. Geom. Anal. 1: 359--372, 1991),
Ding-Jost-Li-Wang (Asian J. Math. 1: 230--248, 1997),
Mancini (J. Geom. Anal. 26: 1202--1230, 2016),
Yang-Zhu (Proc. Amer. Math. Soc. 145: 3953--3959, 2017),
Sun-Zhu (arXiv:2012.12840) and
Zhu (arXiv:2212.09943). Among other things, we prove that the blowup (if happens) must be at the point where the conical angle is the smallest one and $h$ is positive, this is the most important contribution of our paper.
\end{abstract}

\setcounter {section} {0}

\section{Introduction}

Let $(M,g)$ be a compact Riemann surface with unit area, $h$ a smooth function on $M$ which is positive somewhere
 and $\rho>0$ a parameter. In this paper, we consider the singular mean field equation
\begin{align}\label{eq-v}
\Delta v + 4\pi\sum_{i=1}^{\ell}\alpha_i\left(1-\delta_{p_i}\right) = \rho\left(1-\frac{he^v}{\int_Mhe^vd\mu}\right),
\end{align}
where $\Delta$ and $d\mu$ are the Laplace-Beltrami operator and the area element of $(M,g)$ respectively, $p_i\in M$ are given distinct points,
$\alpha_i\in(-1,+\infty)$ and $\delta_{p_i}$ denotes the Dirac measure with pole at $p_i$ for $i=1,\cdots,\ell$.

Since \eqref{eq-v} is invariant under adding a constant, one can assume $\int_M he^{v}d\mu=1$.
When $(M,g)$ is the standard sphere with constant Gauss curvature $4\pi$ (since we have assumed the area of $M$ is unit)
and $\rho=8\pi+4\pi\sum_{i=1}^{\ell}\alpha_i$, the metric $e^vg$ has Gauss curvature
$(4\pi+2\pi\sum_{i=1}^{\ell}\alpha_i)h$ on
$M\setminus\{p_1,\cdots,p_{\ell}\}$ and conical angle $2\pi(1+\alpha_i)$ at $p_i$ for each $i=1,\cdots,\ell$. This is the singular version of
``Nirenberg problem". We refer the reader to \cite{Tro1989, Tro1991, CL-91-JGA, LT92, UY2000, Er2004, CWWX2015, Er2020, De2018, DDI2018, Z2020, Ka2023} for the development
of this topic.
Eq. \eqref{eq-v} also appears in theoretical physics when people describes the models like Abelian Chern-Simons vortices, the reader could find this interpretation in
\cite{HKP1990, JW1990, CY1995, Ta1996, DJLW}. For the study on Eq. \eqref{eq-v} when $\rho\neq\rh$, we refer the reader to \cite{BGJM2019, BT, MR2011}. Besides, we refer
three nice
survey papers \cite{Li08, Ta2010, Lai2016}.

Let $G_p$ be the Green function at $p$ which satisfies
\begin{align}\label{eq-green}
\begin{cases}
\Delta G_p = 1-\delta_p,\\
\int_M G_p d\mu = 0.
\end{cases}
\end{align}
In a normal coordinate system around $p$ we assume that
\begin{align*}
G_p(x)=-\frac{1}{2\pi}\log r +A(p)+\sigma(x),
\end{align*}
where $r(x)=\text{dist}(x,p)$ and $\sigma(x)=O(r)$ as $r\to0$.
By the change of variables
$$u = v+4\pi\sum_{i=1}^{\ell}\alpha_i G_{p_i} := v+h_{\ell},$$
we can transform Eq. (\ref{eq-v}) with $\int_M he^v d\mu=1$ into
\begin{align}\label{eq-u}
\Delta u = \rho\left(1-he^{-h_{\ell}}e^u\right).
\end{align}

To study the existence of Eq. \eqref{eq-u}, one likes to pursue a variational method. Namely, consider
\begin{align*}
J_{\rho}(u)=\frac{1}{2}\int_{M}|\nabla u|^2d\mu + \rho \int_{M}ud\mu
\end{align*}
in the Hilbert space
\begin{align*}
\mathcal{H}=\left\{u\in H^1(M):~\int_M he^ue^{-h_\ell}d\mu=1\right\}.
\end{align*}
Since $h$ is positive somewhere, $\mathcal{H}\neq\emptyset$. By a simple calculation, one knows critical points of $J_{\rho}$ in $\mathcal{H}$ are solutions of \eqref{eq-u}.
Therefore, to solve Eq. \eqref{eq-u}, we can find critical points of $J_\rho$ in $\mathcal{H}$. Let us recall the Moser-Trudinger inequality for surfaces with conical singularities, which has strong relationship with finding critical points for $J_{\rho}$. This inequality was firstly built by Troyanov \cite{Tro1991} for any subcritical case and then sharpened by Chen \cite{Chen1990}. Precisely, there exists a constant $C$
which depends only on $(M,g)$, such that for any $u\in H^1(M)$ with $\int_M |\nabla u|^2 d\mu\leq1$ and $\int_M u e^{-h_{\ell}}d\mu=0$,
\begin{align}\label{Chen-Troyanov inequality}
\int_M e^{4\pi(1+\al)u^2}e^{-h_{\ell}}d\mu\leq C,
\end{align}
where $\al=\min\left\{0,\min_{1\leq i\leq\ell}{\alpha_i}\right\}$. We refer the reader to \cite{Zhu19} for a more general version.
A direct consequence of \eqref{Chen-Troyanov inequality} is
\begin{align}\label{ineq-chentro}
\log\int_M e^{u}e^{-h_{\ell}} d\mu \leq \frac{1}{16\pi(1+\al)}\int_M|\nabla u|^2d\mu + \int_{M}ue^{-h_\ell}d\mu + C.
\end{align}
To study $J_\rho$, a more convenient inequality compared with \eqref{ineq-chentro} was derived by the author \cite{Zhu17}. In fact, we replaced the term $\int_M ue^{-h_\ell}d\mu$ with $\int_M u d\mu$ on the right-hand side of \eqref{ineq-chentro} and obtained
\begin{align}\label{ineq-zhu}
\log\int_M e^{u}e^{-h_{\ell}} d\mu \leq \frac{1}{16\pi(1+\al)}\int_M|\nabla u|^2d\mu + \int_{M}ud\mu + C.
\end{align}
Note that  all the coefficients in \eqref{Chen-Troyanov inequality}-(\ref{ineq-zhu}) are sharp.
Based on this fact, $\rho=8\pi(1+\al):=\rh$ is the critical case in \eqref{eq-u}. When $\rho<\rh$, $J_\rho$ is coercive
in $\mathcal{H}$ and hence it attains its infimum by the standard variational principle (c.f. for example, Theorem 1.2 in \cite{Struwe}). However, when $\rho=\rh$,
the situation becomes more subtle, one just knows $J_\rh$ is bounded from below in $\mathcal{H}$ by \eqref{ineq-zhu}.

When there is no singular source, that is $h_\ell=0,\, \al=0$ and $\rh=8\pi$.
Kazdan and Warner \cite{KW74} asked, under what kind of conditions on $h$, the equation
\begin{align}\label{eq-kw}
 \Delta u = 8\pi-8\pi he^u
\end{align}
has a solution. In the literal, people calls it as \emph{Kazdan-Warner problem}.

In the celebrated paper \cite{DJLW}, Ding, Jost, Li and Wang firstly attacked Kazdan-Warner problem successfully.
If $h$ is positive and $J_{8\pi}$ has no minimum, they proved
\begin{align*}
\inf_{u\in H^1(M)}J_{8\pi} \geq -8\pi-8\pi\log\pi-8\pi\max_{p\in M}\left(4\pi A(p)+\log h(p)\right):=\Lambda_{8\pi}.
\end{align*}
After that, they constructed a blowup sequence $\phi_\ep$ and proved that $J_{8\pi}(\phi_\ep)$ is smaller than $\Lambda_{8\pi}$ for sufficiently small $\ep>0$
under the condition
\begin{align}\label{cond-djlw}
\Delta \log h(p_0) -2K(p_0)+8\pi > 0,
\end{align}
where $p_0$ is the maximum point of $4\pi A(p)+\log h(p)$ on $M$ and $K$ is the Gauss curvature of $(M,g)$. Therefore, under \eqref{cond-djlw}, $J_{8\pi}$ has a minimum
and Eq. \eqref{eq-kw} has a solution. In the following, we call \eqref{cond-djlw} as \emph{Ding-Jost-Li-Wang condition}. Twenty years later,
Yang and the author
\cite{YZ17} generalized this existence result to the case $h\geq0,\,\not\equiv0$, they mainly
excluded the situation that blowup happens at zero point of $h$. Recently, this result was generalized to the case which permits $h$ changing signs.
This were done using variational method first by Sun and Zhu \cite{SZ2022+} and then by the author \cite{Zhu2022+} with a different argument.
The other successful method to study Eq. \eqref{eq-kw}
 is the flow method, we refer the reader to \cite{Ca15}, \cite{LZ19}, \cite{SZ2021}, \cite{WY2022} and \cite{LX2022}.

In this paper, we shall follow Ding-Jost-Li-Wang's method to study Eq. \eqref{eq-u} when there are singular sources and in the critical case $\rho=\rh$.
We call it as the \emph{singular Kazdan-Warner problem}.
Mainly, we remove the positivity restriction on $h$ and just assume $h$ is positive somewhere. Pioneer works were done by Chen-Li \cite{CL-91-JGA} when $M=S^2$
and Mancini \cite{Mancini} for general surfaces and positive $h$.

We consider the perturbed functional $J_{\rho_k}$ when $\rho_k\uparrow \bar\rho$ strictly. Because of (\ref{ineq-zhu}), $J_{\rho_k}$ is coercive in the Hilbert space $\mathcal{H}$. Then by Theorem 1.2 in
\cite{Struwe}, $J_{\rho_k}$ attains its infimum at some $u_k\in\mathcal{H}$. By a direct calculation, we have
\begin{align}\label{eq-uk}
\Delta u_k = \rho_k\left(1-he^{-h_{\ell}}e^{u_k}\right).
\end{align}
We define the \emph{conical singularity} at $p\in M$ as
\begin{align*}
\alpha(p)=
\begin{cases}
\alpha_i,~~~&\text{if~}p=p_i~\text{for~some~}i=1,\cdots,\ell,\\
0,~~~&\text{otherwise.}
\end{cases}
\end{align*}
Correspondingly, the \emph{conical angle} at $p$ is $2\pi(1+\alpha(p))$.
Now we are prepared to state our results.
\begin{thm}\label{thm-zhu-1}
Let $(M,g)$ be a compact Riemann surface with unit area. Denote $M_+=\{x\in M: h(x)>0\}$. If the minimizing sequence $u_k$ of $J_{\rh}$ does not converge in $H^1(M)$, then
\begin{align*}
\inf_{H^1(M)}J_{\rh}(u)=&-\rh\left(1+\log\frac{\pi}{1+\al}\right)\nonumber\\
    &-\rh\max_{p\in M_+, \alpha(p)=\al}\left(4\pi A(p)+\log\left(h(p)\prod_{1\leq i\leq\ell,p_i\neq p}e^{-4\pi\alpha_i G_{p_i}(p)}\right)\right).
\end{align*}
\end{thm}

Concerning the existence result, we have
\begin{thm}\label{thm-zhu-2}
Let $(M,g)$ be a compact Riemann surface with unit area. Denote $M_+=\{x\in M: h(x)>0\}$.
Suppose $h\in C^{\infty}(M)$ is positive somewhere. Then Eq. (\ref{eq-u}) has a solution provided one of the following conditions holds:

$(i)$ $\{p\in M_+: \alpha(p)=\al\}=\emptyset$;

$(ii)$ $\{p\in M_+: \alpha(p)=\al\}\neq\emptyset$ and
\begin{align*}
\inf_{\mathcal{H}}J_{\overline{\rho}}
<&-\rh\left(1+\log\frac{\pi}{1+\al}\right)\nonumber\\
    &-\rh\max_{p\in M_+, \alpha(p)=\al}\left(4\pi A(p)+\log\left(h(p)\prod_{1\leq i\leq\ell,p_i\neq p}e^{-4\pi\alpha_i G_{p_i}(p)}\right)\right)\\
:=&\Lambda_{\rh}.
\end{align*}
\end{thm}

We organize our paper as follows: The introduction and main results are presented in Section 1; In Section 2, we collect some useful lemmas which will be used;
We derive the explicit lower bound of $J_{\rh}$ when $(u_k)$ blows up in Section 3 and construct a blowup sequence in Section 4; In the last section, we complete
the proofs of our main results. Throughout the whole paper, the constant $C$ is varying from line to line and even in the same line, we do not distinguish sequence
and its subsequences since we just consider the existence result.

\section{Some lemmas}

In this section, we would like to present some useful results which will be used in the next section, where we will calculate the explicit lower bound of $J_{\rho_k}$
when $(u_k)$ blows up.

\begin{lem}[Theorem 2.1, \cite{CL-91-JGA}]\label{lem-chenli}
Let $(M,g)$ be a compact Riemann surface with unit area, $\Omega_1$ and $\Omega_2$ be two subsets of $M$ such that $dist(\Omega_1, \Omega_2)\geq\epsilon_0>0$. Assume
$\gamma_0\in(0,\frac{1}{2}]$ is a fixed number. Then for any $\epsilon>0$, there is a constant $C=C(\epsilon_0,\gamma_0,\epsilon)$ such that the inequality
\begin{align*}
\log\int_{M}e^{-h_\ell}e^ud\mu \leq \left(\frac{1}{32\pi(1+\al)}+\epsilon\right)\int_{M}|\nabla u|^2d\mu + \int_M u d\mu+C
\end{align*}
holds for all $u\in H^1(M)$ satisfying
\begin{align*}
\frac{\int_{\Omega_1}e^{-h_\ell}e^ud\mu}{\int_{M}e^{-h_\ell}e^ud\mu}\geq\gamma_0,~~\frac{\int_{\Omega_2}e^{-h_\ell}e^ud\mu}{\int_{M}e^{-h_\ell}e^ud\mu}\geq\gamma_0.
\end{align*}
\end{lem}
The reader can follow Chen-Li's proof effortlessly, we omit it here.


\begin{lem}\label{up-lower-bound}
Suppose $u_k$ attains the infimum of $J_{\rho_k}$ in $\mathcal{H}$. Then there exist two positive constants $c_1$ and $c_2$ such that
\begin{align*}
c_1 \leq \int_{M} e^{-h_\ell}e^{u_k}d\mu \leq c_2.
\end{align*}
\end{lem}
\begin{proof}
Since $u_k\in \mathcal{H}$, one can choose $c_1=1/\max_{M}h$. As to the upper bound, notice
$$J_{\rho_k}(u_k) = \inf_{\mathcal{H}}J_{\rho_k}(u)\leq C,$$
then the Moser-Trudinger inequality (\ref{ineq-zhu}) and Jensen's inequality yield that
\begin{align*}
\log \int_{M} e^{-h_\ell}e^{u_k}d\mu
\leq&\frac{1}{\overline{\rho}}J_{\rho_k}(u_k)+(1-\frac{\rho_k}{\overline{\rho}})\int_{M}u_kd\mu +C\\
\leq& (1-\frac{\rho_k}{\overline{\rho}})\int_{M}(-h_{\ell}+u_k)d\mu+(1-\frac{\rho_k}{\overline{\rho}})\int_{M}h_{\ell}d\mu+C\\
\leq& (1-\frac{\rho_k}{\overline{\rho}})\log\left(\int_{M} e^{-h_\ell}e^{u_k}d\mu\right)+C,
\end{align*}
where in the last inequality we have used the fact that $h_{\ell}\in L^1(M)$. Now we finish the proof.
\end{proof}

\begin{lem}\label{lem-Lq}
If $(u_k)$ blows up, then for every $q\in(1,2)$, it holds that $$\|\nabla u_k\|_{L^q(M)}\leq C.$$
\end{lem}
\begin{proof}
Let $q'=\frac{q}{q-1}>2$. By definition, we have
\begin{align*}
\|\nabla u_k\|_{L^q(M)}\leq\sup\left\{\left|\int_M \nabla u_k\cdot\nabla\varphi d\mu\right|:\varphi\in W^{1,q'}(M),\int_M\varphi d\mu=0,\|\varphi\|_{W^{1,q'}(M)}=1\right\}.
\end{align*}
It follows form the Sobolev embedding theorem that
$$\|\varphi\|_{L^{\infty}(M)}\leq C.$$
Then by equation (\ref{eq-uk}) and Lemma \ref{up-lower-bound} we obtain
\begin{align*}
\left|\int_M \nabla u_k\cdot\nabla\varphi d\mu\right|
&= \left|-\int_M \Delta u_k \varphi d\mu\right|\\
&= \left|\int_M \rho_k\left(he^{-h_{\ell}}e^{u_k}-1\right)\varphi d\mu\right|\\
&\leq \overline{\rho}\|\varphi\|_{L^{\infty}(M)}\left(\max_{M}|h|\int_Me^{-h_{\ell}}e^{u_k}d\mu+1\right)\\
&\leq C.
\end{align*}
This gives the proof of the lemma.
\end{proof}

Denote $\overline{u_k}=\int_M u_k d\mu$, $\lambda_k=\max\limits_{M}u_k$. Assume $u_k(x_k)=\lambda_k$ for some $x_k\in M$ and
$x_k\to p\in M$ as $k\to\infty$.
\begin{lem}\label{lem-kehua}
The following three conditions are equivalent:

$(i)$ $\lambda_k\to +\infty$ as $k\to\infty$;

$(ii)$ $\|\nabla u_k\|_{L^2(M)}\to+\infty$ as $k\to\infty$;

$(iii)$ $\overline{u_k}\to-\infty$ as $k\to\infty$.
\end{lem}
\begin{proof}
Since $J_{\rho_k}(u_k)$ is bounded, we have $(ii)\Leftrightarrow(iii)$ obviously.

$(i)\Rightarrow(ii)$: Suppose not, $\|\nabla u_k\|_{L^2(M)}\leq C$, then by the proved equivalent relation $(ii)\Leftrightarrow(iii)$, $\overline{u_k}$
is bounded from below. It follows by Jensen's inequality and Lemma \ref{up-lower-bound} that
\begin{align*}
\overline{u_k}-\overline{h_{\ell}} \leq \log\left(\int_M e^{-h_{\ell}+u_k}d\mu\right)\leq C,
\end{align*}
this together with the fact $h_{\ell}\in L^1(M)$ yields
$\overline{u_k} \leq C$.
By Poincar\'{e}'s inequality,
\begin{align*}
\int_M u_k^2 d\mu-\overline{u_k}^2 = \int_M(u_k-\overline{u_k})^2d\mu\leq C\int_{M}|\nabla u_k|^2d\mu\leq C.
\end{align*}
So $(u_k)$ is bounded in $L^2(M)$ and therefore it is in $W^{2,2}(M)$. Then elliptic estimates tell us that $\|u_k\|_{L^{\infty}(M)}$ is bounded.
Of course, we also have $\lambda_k\leq C$.

$(ii)\Rightarrow(i)$: Suppose not, $\lambda_k\leq C$, then $e^{u_k}$ is bounded.
It is clear that $he^{-h_{\ell}}$ is bounded in $L^s(M)$ for some $s>1$ ($s=\infty$ if $\overline{\alpha}=0$
and $s<-1/\overline{\alpha}$ if $\overline{\alpha}<0$). By Lemma \ref{lem-Lq}, $\|u_k-\overline{u_k}\|_{L^{s_1}(M)}$ for any $s_1>1$. Since $\Delta(u_k-\overline{u_k})$
is bounded in $L^s(M)$, it follows from the elliptic estimates that $u_k-\overline{u_k}$ is bounded in $L^{\infty}(M)$. This together with $\overline{u_k}\to-\infty$
yields that
\begin{align*}
\lim_{k\to\infty}\int_M he^{-h_{\ell}}e^{u_k}d\mu=\lim_{k\to\infty}e^{\overline{u_k}}\int_{M}he^{-h_{\ell}}e^{u_k-\overline{u_k}}d\mu=0.
\end{align*}
It contradicts $u_k\in\mathcal{H}$. This finishes the proof.
\end{proof}

By Brezis-Merle's lemma (\cite{BM}, Theorem 1) and following elliptic estimates as the proof of Lemma 2.8 in \cite{DJLW} (or Lemma 2.9 in \cite{Zhu17}), one has
\begin{lem}\label{lem-bm}
Let $\Omega\subset M$ be a domain. If
\begin{align*}
\int_{\Omega} |h|e^{-h_{\ell}+u_k}d\mu \leq \frac{1}{2(1+\overline{\alpha})}-\delta
\end{align*}
for some $0<\delta<\frac{1}{2(1+\overline{\alpha})}$, then
\begin{align*}
\|u_k-\overline{u_k}\|_{L^{\infty}_{\text{loc}}(\Omega)}\leq C.
\end{align*}
\end{lem}

Due to Lemma \ref{lem-bm}, we define the \emph{blowup set} of $(u_k)$ as
\begin{align}\label{defi-blowupset}
S = \left\{x\in M: \lim_{r\to0}\lim_{k\to\infty}\int_{B_r(x)}|h|e^{-h_{\ell}+u_k}d\mu\geq\frac{1}{2(1+\overline{\alpha})}\right\}.
\end{align}

The following observation is very important to us. Even though $h$ may change signs, the blowup set is still a single point set at most. It breaks the
surmise that when $h$ changes signs, it may has two or more blowup points, they concentrate at different points where $h$ may be negative and positive,
but cancel each other.

\begin{lem}\label{lem-blowupset}
$S = \{p\}$.
\end{lem}
\begin{proof}
We divide the whole proof into three parts: (1) $S\neq\emptyset$;  (2) $\# S=1$; (3) $S = \{p\}$. In the following, we shall prove these three parts on by one.

(1) $S\neq\emptyset$. Suppose not, for every $x\in M$, there exists a positive number $r_x$ which is smaller then the injective radius of $M$, such that
\begin{align*}
\int_{B_{r_x}(x)} |h|e^{-h_{\ell}+u_k}d\mu < \frac{1}{2(1+\overline{\alpha})}.
\end{align*}
Then by Lemma \ref{lem-bm} we obtain that
\begin{align*}
\|u_k-\overline{u_k}\|_{L^{\infty}(B_{r_x/2}(x))}\leq C.
\end{align*}
This combining with a finite covering argument tells us that
$$\|u_k-\overline{u_k}\|_{L^{\infty}(M)}\leq C.$$
Since $(u_k)$ blows up, we have $\overline{u_k}\to-\infty$ by Lemma \ref{lem-kehua} and then we have $u_k\to-\infty$ as $k\to\infty$. This contradicts with Lemma \ref{up-lower-bound}.

(2) $\# S=1$. Suppose not, we must have $\# S\geq2$ since we have proved that $S\neq\emptyset$. Suppose $x_1\neq x_2 \in S$, then for sufficiently small $r$, one has
\begin{align*}
\frac{\int_{B_{r}(x_i)}e^{-h_{\ell}+u_k}d\mu}{\int_M e^{-h_{\ell}+u_k}d\mu}\geq \frac{1}{2(1+\overline{\alpha})c_2\|h\|_{L^{\infty}(M)}},~~i=1,2.
\end{align*}
Then by Lemma \ref{lem-chenli} we have for $\epsilon=\frac{1}{8\overline{\rho}}$, there exists a constant $C$ such that
\begin{align}\label{lem-blowupset_1}
\log\int_M e^{-h_{\ell}+u_k}d\mu
&\leq \frac{3}{8\overline{\rho}}\int_{M}|\nabla u_k|^2d\mu+\overline{u_k}+C\nonumber\\
&= \frac{3}{4\overline{\rho}}J_{\rho_k}(u_k)+\left(1-\frac{3\rho_k}{4\overline{\rho}}\right)\overline{u_k}+C\nonumber\\
&\to-\infty~~\text{as}~k\to\infty,
\end{align}
where we have used facts $J_{\rho_k}(u_k)$ is bounded and $\overline{u_k}\to\infty$ as $k\to\infty$. (\ref{lem-blowupset_1}) contradicts Lemma \ref{up-lower-bound}.
Therefore, $\# S=1$.

(3) $S = \{p\}$. Let us recall that $u_k(x_k)=\max\limits_{M}u_k$ and $x_k\to p$ as $k\to\infty$. If $p\notin S$, then by (\ref{defi-blowupset}) (the definition of $S$)
there exist $0<\delta<\frac{1}{2(1+\overline{\alpha})}$ and $r>0$ such that
\begin{align*}
\int_{B_r(p)}|h|e^{-h_{\ell}+u_k}d\mu<\frac{1}{2(1+\overline{\alpha})}-\delta.
\end{align*}
Then Lemma \ref{lem-bm} yields that
\begin{align*}
\|u_k-\overline{u_k}\|_{L^{\infty}(B_{r/2}(p))}\leq C.
\end{align*}
So we have by Lemma \ref{lem-kehua} that
\begin{align*}
u_k(x_k)\leq \overline{u_k}+C\to-\infty~~\text{as}~k\to\infty,
\end{align*}
this contradicts $u_k(x_k)=\lambda_k\to+\infty$ as $k\to\infty$. Therefore, $S=\{p\}$.
\end{proof}

\begin{lem}\label{lem-hp>0}
If $(u_k)$ blows up, then $h$ must be positive at the single blowup point, i.e., $h(p)>0$.
\end{lem}
\begin{proof}
It follows from Lemmas \ref{lem-bm} and \ref{lem-blowupset} that
$$\lim_{r\to0}\lim_{k\to\infty}\int_{M\setminus B_r(p)}e^{-h_{\ell}+u_k}d\mu=0$$
and then
\begin{align*}
h(p)\lim_{r\to0}\lim_{k\to\infty}\int_{B_r(p)}e^{-h_{\ell}+u_k}d\mu=1.
\end{align*}
So $h(p)>0$ and we finish the proof of the lemma.
\end{proof}


Recall that the author \cite{Zhu17} proved, if $(u_k)$ blows up, it must blow up at the point where
the conical angle is the smallest, i.e., $\alpha(p)=\overline{\alpha}$.
Now, let us recall more results in \cite{Zhu17}.
We choose an isothermal coordinate system around $p$, denote $r_k=e^{-\frac{\lambda_k}{2(1+\overline{\alpha})}}$, then
\begin{align}\label{eq-bubble}
\varphi_k(x):=u_k(x_k+r_kx)-\lambda_k\to\varphi(x)=-2\log\left(1+\frac{\pi}{1+\overline{\alpha}}H(p)|x|^{2(1+\overline\alpha)}\right),
\end{align}
in $C_{\text{loc}}^1(\mathbb{R}^2)$ if $\overline{\alpha}=0$ and in $C_{\text{loc}}^1(\mathbb{R}^2\setminus\{0\})\cap C^0_{\text{loc}}(\mathbb{R}^2)\cap W^{2,s}_{\text{loc}}(\mathbb{R}^2)$ for every $s\in(1,-1/\overline{\alpha})$ if $\overline{\alpha}<0$ as $k\to\infty$,
where $$H(p)=h(p)e^{-4\pi A(p)}\prod_{1\leq i\leq\ell, p_i\neq p}e^{-4\pi\alpha_iG_{p_i}(p)}.$$
Calculate directly, one has
\begin{align}\label{int=1}
\lim_{R\to+\infty}\lim_{k\to\infty}\int_{B_{Rr_k}(x_k)}he^{-h_{\ell}+u_k}d\mu=\int_{\mathbb{R}^2}H(p)|x|^{2\overline{\alpha}}e^{\varphi_0}dx=1.
\end{align}
Since $S=\{p\}$, for any $x\in M\setminus\{p\}$, there exists a $\gamma_x\in(0,1/2)$ and a small $r_x\in(0,\frac{1}{2(1+\overline{\alpha})}\text{dist}(x,p))$ such that
$$\int_{B_{r_x}(x)}|h|e^{-h_{\ell}+u_k}d\mu<\frac{1}{2(1+\overline{\alpha})}-\gamma_x.$$
By Lemma \ref{lem-bm}, $\|u_k-\overline{u_k}\|_{L^{\infty}_{B_{r_x/2}(x)}}\leq C$, then by Lemma \ref{lem-kehua} we have
$u_k(x)\leq C+\overline{u_k}\to-\infty$ as $k\to\infty$. So for any $\Omega\subset\subset M\setminus\{p\}$, there holds
\begin{align}\label{int=0}
\int_{\Omega}|h|e^{-h_{\ell}+u_k}d\mu\to0~~\text{as}~~k\to\infty.
\end{align}
By (\ref{int=1}) and (\ref{int=0}) we get that $he^{-h_{\ell}+u_k}$ converges to $\delta_p$ in the sense of measure. Therefore $u_k-\overline{u_k}\to \overline{\rho}G_p(x)$ weakly in $W^{1,q}(M)$ for any $1<q<2$, where $G_p$ is the Green function satisfying (\ref{eq-green}), since $G_p$ is the only solution of (\ref{eq-green}) in $W^{1,q}(M)$. Lemma \ref{lem-bm}
and (\ref{int=0}) yield that for any $\Omega\subset\subset M\setminus\{p\}$,
\begin{align*}
\|u_k-\overline{u_k}\|_{L^{\infty}(\Omega)}\leq C.
\end{align*}
This inequality together with the standard elliptic estimates yields that
\begin{align}\label{convergence}
u_k-\overline{u_k}\to \rh G_p~~\text{in}~~C_{\text{loc}}^{\gamma}(M\setminus\{p\})\cap W^{1,s}(M)~~\text{as}~~k\to\infty
\end{align}
for some $\gamma\in(0,1)$ and $s>2$.

\section{Lower bound of $J_{\rh}$ when $(u_k)$ blows up}

Based on the prepared work in the former section, we could estimate the explicit lower bound of $J_{\rh}$ in this section.

Since the asymptotic phenomenons are different, we divide the integral $\int_M|\nabla u_k|^2$ into three parts. Namely,
\begin{align}\label{est-integral}
\int_M|\nabla u_k|^2  = \int_{M\setminus B_{\delta}(x_k)}|\nabla u_k|^2  + \int_{B_{\delta}(x_k)\setminus B_{Rr_k}(x_k)}|\nabla u_k|^2 +\int_{B_{Rr_k}(x_k)}|\nabla u_k|^2.
\end{align}
For the first term in the right-hand side of (\ref{est-integral}), we have
\begin{align}\label{est-int-out}
\int_{M\setminus B_{\delta}(x_k)}|\nabla u_k|^2
=&\int_{M\setminus B_{\delta}(p)}|\nabla G_p|^2+o_k(1)\nonumber\\
=&-\frac{1}{2\pi}\rh^2\log\delta+\rh^2A(p)+o_k(1)+o_\delta(1),
\end{align}
where and in the following we use $o_k(1)$ (resp. $o_R(1)$; $o_\delta(1)$) to denote the terms which tend to $0$ as $k\to\infty$
 (resp. $R\to\infty$; $\delta\to0$).

For the third term in the right-hand side of (\ref{est-integral}), we have
\begin{align}\label{est-int-in}
\int_{B_{Rr_k}(x_k)}|\nabla u_k|^2
=&\int_{\mathbb{B}_{R(0)}}|\nabla_{\R^2}\varphi|^2+o_k(1)\nonumber\\
=&2\rh\log(1+\frac{\pi H(p)}{1+\overline{\alpha}}R^{2(1+\overline{\alpha})})-2\rh +o_k(1)+o_R(1).
\end{align}

We use the capacity method, to estimate the second term in the right-hand side of (\ref{est-integral}).
This method was first used by Li \cite{Li01} to estimate the neck-part of the integral $\int_M|\nabla u_k|^2$ when he studied the Moser-Trudinger inequality
in dimension two.
Set
\begin{align*}
u_k^*(r) = \frac{1}{2\pi}\int_0^{2\pi}u_k(x_k+re^{i\theta})d\theta.
\end{align*}
Then it follows by (\ref{convergence}) and (\ref{eq-bubble}) that
\begin{align*}
u_k^*(\delta) =& \uk+\rh(-\frac{1}{2\pi}\log\delta+A(p))+o_k(1)+o_{\delta}(1).\\
u_k^*(Rr_k)=& \lam_k-2\log(1+\frac{\pi H(p)}{1+\overline{\alpha}}R^{2(1+\overline{\alpha})})+o_k(1)+o_R(1).
\end{align*}
Suppose $w_k$ solves
\begin{align*}
\begin{cases}
\Delta_{\R^2}w_k=0~~&\text{in}~\mathbb{B}_{\delta}(0)\setminus\mathbb{B}_{Rr_k}(0),\\
w_k(x)=u_k^*(x)~~&\text{on}~\partial \left(\mathbb{B}_{\delta}(0)\setminus\mathbb{B}_{Rr_k}(0)\right),
\end{cases}
\end{align*}
then
$$w_k(x)=\frac{u_k^*(\delta)\left(\log|x|-\log(Rr_k)\right)+u_k^*(Rr_k)\left(\log\delta-\log|x|\right)}{\log\delta-\log(Rr_k)}.$$
So we have
\begin{align}\label{est-int-neck}
&\int_{B_{\delta}(x_k)\setminus B_{Rr_k}(x_k)}|\nabla u_k|^2\no\\
\geq& \int_{\mathbb{B}_{\delta}(0)\setminus \mathbb{B}_{Rr_k}(0)} |\nabla_{\R^2}w_k(x)|dx\no\\
=&2\pi\frac{\left(u_k^*(\delta)-u_k^*(Rr_k)\right)^2}{\log\delta-\log(Rr_k)}\no\\
=&2\pi\frac{\left(\uk-\lam_k+\frac{\rh}{2\pi}\log\frac{R}{\delta}+\rh A(p)+2\log\frac{\pi H(p)}{1+\al}+o(1)\right)^2}{\frac{\lam_k}{2(1+\al)}+\log\frac{R}{\delta}}.
\end{align}
Combining (\ref{est-int-out}), (\ref{est-int-in}) and (\ref{est-int-neck}) with the fact that $J_{\rho_k}(u_k)=\inf_{\mathcal{H}}J_{\rho_k}$ is bounded, we obtain that
\begin{align}\label{est_1}
J_{\rho_k}(u_k)\geq&\frac{\rh^2}{2}A(p)+\rh\log\frac{\pi H(p)}{1+\al}-\rh+\frac{\rh^2}{4\pi}\log\frac{R}{\delta}+\rho_k\uk\no\\
&+\pi\frac{\left(\uk-\lam_k+\frac{\rh}{2\pi}\log\frac{R}{\delta}+\rh A_p+2\log\frac{\pi H(p)}{1+\al}+o(1)\right)^2}{\frac{\lam_k}{2(1+\al)}+\log\frac{R}{\delta}}
\end{align}
is bounded, dividing the quantity on the right-hand side of \eqref{est_1} by $\lam_k$ and letting $k$ tend to $\infty$ one arrives at
\begin{align*}
\lim_{k\to\infty}\left(\frac{\uk}{\lam_k}-1+\frac{\frac{\rh}{2\pi}\log\frac{R}{\delta}+\rh A(p)+2\log\frac{\pi H(p)}{1+\al}+o(1)}{\lam_k}+\frac{2\rho_k}{\rh}\right)^2=0.
\end{align*}
Here and in the following we use $o(1)$ to denote terms which tend to $0$ as $k\to\infty$ first and then $R\to\infty$ and $\delta\to0$.
Hence
\begin{align*}
\uk = (1-\frac{2\rho_k}{\rh})\lam_k-\left(\frac{\rh}{2\pi}\log\frac{R}{\delta}+\rh A(p)+2\log\frac{\pi H(p)}{1+\al}+o(1)\right).
\end{align*}
Taking this into (\ref{est_1}) and calculating directly we have
\begin{align*}
J_{\rho_k}(u_k)\geq&-\frac{\rh^2}{2}A(p)-\rh\log\frac{\pi H(p)}{1+\al}-\rh\\
&+\frac{\rho_k}{\rh}(\rh-\rho_k)\lam_k+\frac{1}{4\pi}(\rh-\rho_k)^2\log\frac{R}{\delta}+o(1).
\end{align*}
Then similar as Lemma 2.10 in \cite{DJLW}
we have when $(u_k)$ blows up,
\begin{align}\label{est_2}
\inf_{\mathcal{H}}J_{\rh}(u)\geq&\lim_{\delta\to0}\lim_{R\to\infty}\lim_{k\to\infty}J_{\rho_k}(u_k)
\geq-\frac{\rh^2}{2}A(p)-\rh\log\frac{\pi H(p)}{1+\al}-\rh\no\\
\geq&-\rh\left(1+\log\frac{\pi}{1+\al}\right)\nonumber\\
    &-\rh\max_{p\in M_+, \alpha(p)=\al}\left(4\pi A(p)+\log\left(h(p)\prod_{1\leq i\leq\ell,p_i\neq p}e^{-4\pi\alpha_i G_{p_i}(p)}\right)\right).
\end{align}

\section{The blowup sequence}

Since we have derived an explicit lower bound of $J_{\rh}$ when $(u_k)$ blows up in the former section, the successful experiences in \cite{Sch84, ES86, DJLW} tell us that,
if we can construct a blowup sequence $\phi_\ep$ which can make $J_{\rh}(\phi_\ep)<\Lambda_{\rh}$ for sufficiently small $\ep>0$, then the blowup will not happen and $J_{\rh}$ attains is infimum.

Let $p\in M$ be such that $\alpha(p)=\overline{\alpha}$ and
\begin{align*}
&4\pi A(p)+\log\left(h(p)\prod_{1\leq i\leq\ell,p_i\neq p}e^{-4\pi\alpha_i G_{p_i}(p)}\right)\nonumber\\
=&\max_{q\in M_+, \alpha(q)=\al}\left(4\pi A(q)+\log\left(h(q)\prod_{1\leq i\leq\ell,p_i\neq q}e^{-4\pi\alpha_i G_{p_i}(q)}\right)\right).
\end{align*}
Let $(\Om;(x_1,x_2))$ be an isothermal coordinate system around $p$ and set
$$r(x_1,x_2)=\sqrt{x_1^2+x_2^2},~~~~\text{and}~~B_{\delta}(p)=\{(x_1,x_2): r(x_1,x_2)< \delta\}.$$
We write near $p$ the metric
$$g|_{\Om}=e^{\psi(x_1,x_2)}(dx_1^2+dx_2^2)$$
with $\psi(x_1,x_2)=O(r)~(r\to0)$. It is well known that
\begin{align*}
|\nabla u|^2d\mu=|\nabla u|^2dx_1dx_2~~~\text{and}~~~
\frac{\p u}{\p n}ds_g=\frac{\p u}{\p r}rd\theta~~\text{on}~\p B_{r}(p).
\end{align*}

Denote $\ga=\frac{\ep^{-\frac{1}{2(1+\al)}}}{-\log\ep}$ and $r_{\epsilon}:=\gamma_{\epsilon}{\epsilon}^{\frac{1}{2(1+\al)}}$. We define
\begin{align*}
\phi_\epsilon=
\begin{cases}
-2\log\left(\epsilon+r^{2(1+\overline{\alpha})}\right)+\log\ep~~~&\text{if}~r\leq r_\epsilon,\\
\overline{\rho}\left(G_p-\eta\sigma\right)+C_{\epsilon}+\log\epsilon~~~&\text{if}~r\geq r_{\epsilon},
\end{cases}
\end{align*}
where $r=\text{dist}(x,p)$, $\eta\in C^{1}_0(B_{2r_\ep}(p))$ is a radial cutoff function which satisfies $\eta\equiv1$ in $B_{r_\ep}(p)$ and $\mid\nabla\eta\mid\leq Cr_\ep^{-1}$, and
$$C_\epsilon = -2\log\frac{1+\gamma_{\epsilon}^{2(1+\overline{\alpha})}}{\gamma_{\epsilon}^{2(1+\overline{\alpha})}} - \overline{\rho}A(p).$$


Now, by direct calculations
\begin{align}\label{phi-in}
\int_{B_{r_\ep}(p)}|\nabla \phi_{\ep}|^2 d\mu=& 2\rh \log(1+\ga^{2(1+\al)})-2\rh+o_{\ep}(1).
\end{align}

\begin{align}\label{phi-out}
\int_{M\setminus B_{r_\ep}(p)}|\nabla\phi_\ep|^2d\mu
=&\rh^2\int_{M\setminus B_{r_\ep}(p)}|\nabla(G_p-\eta\beta)|^2d\mu\no\\
=&\rh^2\int_{M\setminus B_{r_\ep}(p)}|\nabla G_p|^2d\mu + \rh^2\int_{B_{2r_{\ep}}(p)\setminus B_{r_\ep}(p)}|\nabla(\eta\beta)|^2d\mu\no\\
 &-2\rh^2\int_{B_{2r_\ep}(p)\setminus B_{r_\ep}(p)}\nabla G_p\cdot\nabla(\eta\beta) d\mu.
\end{align}
Do calculations, one has directly
\begin{align}\label{phi-out-1}
\int_{M\setminus B_{r_\ep}(p)}|\nabla G_p|^2d\mu
=&-\int_{M\setminus B_{r_\ep}(p)}G_p\Delta G_p d\mu-\int_{\p B_{r_\ep}(p)}G_p\frac{\p G_p}{\p n}ds_g\no\\
=&-\frac{1}{4\pi}\log r_{\ep}^2+A(p)+o_{\ep}(1),
\end{align}
\begin{align}\label{phi-out-2}
\int_{B_{2r_\ep(p)}\setminus B_{r_\ep}(p)}|\nabla(\eta\sigma)|^2d\mu=o_{\ep}(1)
\end{align}
and
\begin{align}\label{phi-out-3}
&-2\int_{B_{2r_\ep(p)}\setminus B_{r_\ep}(p)}\nabla G_p\cdot\nabla(\eta\beta) d\mu\no\\
=&\int_{B_{2r_\ep(p)}\setminus B_{r_\ep}(p)} \eta\sigma \Delta G_p d\mu-\int_{\p \left(B_{2r_\ep(p)}\setminus B_{r_\ep}(p)\right)} \eta\sigma \frac{\p G_p}{\p n}ds_g\no\\
=&\int_{B_{2r_\ep(p)}\setminus B_{r_\ep}(p)} \eta\sigma  d\mu + \int_{\p B_{r_\ep}(p)} \sigma \frac{\p G_p}{\p r} ds_g = o_\ep(1).
\end{align}
Substituting (\ref{phi-out-1})-(\ref{phi-out-3}) into \eqref{phi-out} and then together with \eqref{phi-in} we have
\begin{align}\label{phi-last}
\int_{M}|\nabla\phi_\ep|^2d\mu = -2\rh\log\ep-2\rh+\rh^2 A(p)+o_\ep(1).
\end{align}
Calculating directly, one has
\begin{align}\label{mean-in}
&\int_{B_{r_\ep}(p)}-2\log(\ep+r^{2(1+\al)})d\mu\no\\
=&-2\pi r_{\ep}^2\log(\ep+r_{\ep}^{2(1+\al)})
 +\frac{\rh}{2}\int_0^{r_{\ep}}\frac{r^{3+2\al}}{\ep+r^{2(1+\al)}}dr\no\\
  &+ O(r_{\ep}^4\log(\ep+r_{\ep}^{2(1+\al)}))\no\\
=&o_{\ep}(1)
\end{align}
since
\begin{align*}
0<\int_0^{r_{\ep}}\frac{r^{3+2\al}}{\ep+r^{2(1+\al)}}dr\leq\frac{1}{2}r_\ep^2.
\end{align*}
It is clear that
\begin{align}\label{mean-out}
 &\int_{M\setminus B_{r_\ep}(p)}\phi_\ep d\mu\no\\
=&\rh \int_{M\setminus B_{r_\ep}(p)}G_p -\rh \int_{B_{2r_\ep}(p)\setminus B_{r_\ep}(p)}\eta\sigma + C_{\ep}(1-\text{Vol}(B_{r_\ep}(p)))\no\\
 &+\log\ep(1-\text{Vol}(B_{r_\ep}(p)))\no\\
=&C_{\ep}(1-\text{Vol}(B_{r_\ep}(p)))+\log\ep(1-\text{Vol}(B_{r_\ep}(p)))+o_{\ep}(1).
\end{align}
By combining \eqref{mean-in} and \eqref{mean-out}, we obtain that
\begin{align}\label{mean}
 &\int_{M}\phi_\ep d\mu = \log\ep-\rh A(p)+o_{\ep}(1).
\end{align}
We have
\begin{align}\label{log-in}
&\int_{B_{r_\ep}(p)} e^{-4\pi\al G_p}e^{\phi_\ep} d\mu\no\\
=&e^{-4\pi\al A(p)}\int_0^{2\pi}\int_0^{r_\ep} \frac{\ep r^{2\al}}{(\ep+r^{2(1+\al)})^2}e^{-4\pi\al\sigma+\psi} rdrd\theta\no\\
=&e^{-4\pi\al A(p)}\int_0^{r_\ep} \frac{\ep r^{2\al}}{(\ep+r^{2(1+\al)})^2}(2\pi+O(r^2)) rdr\no\\
=&\frac{\pi}{1+\al}e^{-4\pi\al A(p)}\frac{\ga^{2(1+\al)}}{1+\ga^{2(1+\al)}} + \pi e^{-4\pi\al A(p)}O\left(\int_0^{r_\ep}\frac{\ep r^{3+2\al}}{(\ep+r^{2(1+\al)})^2}dr\right)\no\\
=&\frac{\pi}{1+\al}e^{-4\pi\al A(p)}+o_{\ep}(1),
\end{align}
where we have used
\begin{align*}
\int_{0}^{2\pi} e^{-4\pi\al\sigma+\psi} d\theta = 2\pi+O(r^2)
\end{align*}
and
\begin{align*}
0<\int_0^{r_{\ep}}\frac{r^{3+2\al}}{(\ep+r^{2(1+\al)})^2}dr\leq\frac{1}{-2\al}r_\ep^{-2\al}.
\end{align*}
in the second equality and the last equality respectively.

By choosing $\delta>0$ sufficiently small we can make the expansion of $G_p$ hold in $B_{\delta}(p)$, then
\begin{align}\label{log-out}
&\int_{M\setminus B_{r_{\ep}(p)}}e^{-4\pi\al G_p}e^{\phi_\ep}d\mu\no\\
=&\ep\int_{M\setminus B_{\delta}(p)} e^{4\pi(2+\al)G_p+C_\ep}d\mu
+\ep\int_{B_\delta(p)\setminus B_{r_{\ep}}(p)}e^{4\pi(2+\al)G_p-\rh \eta \sigma+C_\ep}d\mu\no\\
=&\ep e^{C_\ep} \int_{r_\ep}^{\delta} r^{-2(2+\al)}(2\pi+O(r^2))rdr+o_\ep(1)\no\\
=&o_{\ep}(1),
\end{align}
where we have used the fact that $\gamma^{-2(1+\al)}=o_{\ep}(1)$ in the last equality.
Combining \eqref{log-out} with \eqref{log-in} we have
\begin{align*}
\int_{M}e^{-4\pi\al G_p}e^{\phi_\ep}=\frac{\pi}{1+\al}e^{-4\pi\al A(p)}+o_\ep(1).
\end{align*}
It is clear that
\begin{align*}
\int_M he^{-h_{\ell}}e^{\phi_\ep}
 =& h(p)\prod_{1\leq i\leq\ell, p_i\neq p}e^{-4\pi\alpha_iG_{p_i}(p)}\int_{M}e^{-4\pi\al G_p}e^{\phi_\ep}\\
  &+\int_{M}\left(h\prod_{1\leq i\leq\ell, p_i\neq p}e^{-4\pi\alpha_iG_{p_i}}-h(p)\prod_{1\leq i\leq\ell, p_i\neq p}e^{-4\pi\alpha_iG_{p_i}(p)}\right)e^{-4\pi\al G_p}e^{\phi_\ep}
\end{align*}
By direct calculations, we have
\begin{align*}
&\int_{B_{r_\ep}}\left(h\prod_{1\leq i\leq\ell, p_i\neq p}e^{-4\pi\alpha_iG_{p_i}}-h(p)\prod_{1\leq i\leq\ell, p_i\neq p}e^{-4\pi\alpha_iG_{p_i}(p)}\right)e^{-4\pi\al G_p}e^{\phi_\ep}\\
=&O(1)\int_0^{r_\ep}\frac{\ep r^{3+2\al}}{(\ep+r^{2(1+\al)})^2}dr+o_{\ep}(1)=o_{\ep}(1)
\end{align*}
and
\begin{align*}
&\int_{M\setminus B_{r_\ep}(p)}\left(h\prod_{1\leq i\leq\ell, p_i\neq p}e^{-4\pi\alpha_iG_{p_i}}-h(p)\prod_{1\leq i\leq\ell, p_i\neq p}e^{-4\pi\alpha_iG_{p_i}(p)}\right)e^{-4\pi\al G_p}e^{\phi_\ep}\\
=&\int_{B_{\delta}(p)\setminus B_{r_\ep}(p)}\left(h\prod_{1\leq i\leq\ell, p_i\neq p}e^{-4\pi\alpha_iG_{p_i}}-h(p)\prod_{1\leq i\leq\ell, p_i\neq p}e^{-4\pi\alpha_iG_{p_i}(p)}\right)e^{-4\pi\al G_p}e^{\phi_\ep}\\
&+\int_{M\setminus B_{\delta}(p)}\left(h\prod_{1\leq i\leq\ell, p_i\neq p}e^{-4\pi\alpha_iG_{p_i}}-h(p)\prod_{1\leq i\leq\ell, p_i\neq p}e^{-4\pi\alpha_iG_{p_i}(p)}\right)e^{-4\pi\al G_p}e^{\phi_\ep}\\
=&o_{\ep}(1).
\end{align*}
Hence
\begin{align*}
\int_{M}\left(h\prod_{1\leq i\leq\ell, p_i\neq p}e^{-4\pi\alpha_iG_{p_i}}-h(p)\prod_{1\leq i\leq\ell, p_i\neq p}e^{-4\pi\alpha_iG_{p_i}(p)}\right)e^{-4\pi\al G_p}e^{\phi_\ep}=o_{\ep}(1).
\end{align*}
Therefore, we have
\begin{align*}
&\int_M he^{-h_{\ell}}e^{\phi_\ep}\\
=& h(p)\prod_{1\leq i\leq\ell, p_i\neq p}e^{-4\pi\alpha_iG_{p_i}(p)}\int_{M}e^{-4\pi\al G_p}e^{\phi_\ep}\\
  &+\int_{M}\left(h\prod_{1\leq i\leq\ell, p_i\neq p}e^{-4\pi\alpha_iG_{p_i}}-h(p)\prod_{1\leq i\leq\ell, p_i\neq p}e^{-4\pi\alpha_iG_{p_i}(p)}\right)e^{-4\pi\al G_p}e^{\phi_\ep}\\
=&h(p)\prod_{1\leq i\leq\ell, p_i\neq p}e^{-4\pi\alpha_iG_{p_i}(p)}\left[\frac{\pi}{1+\al}e^{-4\pi\al A(p)}+o_{\ep}(1)\right].
\end{align*}
Then
\begin{align}\label{log-last}
\log\int_M he^{-h_{\ell}}e^{\phi_\ep}=\log\left(\frac{\pi}{1+\al}e^{-4\pi\al A(p)}h(p)\prod_{1\leq i\leq\ell, p_i\neq p}e^{-4\pi\alpha_iG_{p_i}(p)}\right)+o_{\ep}(1).
\end{align}
We put \eqref{phi-last}, \eqref{mean} and \eqref{log-last} together and obtain that
\begin{align}\label{bound}
\lim_{\ep\to0}J_{\rh}(\phi_\ep)
=&-\rh\left(1+\log\frac{\pi}{1+\al}\right)\nonumber\\
    &-\rh\max_{q\in M_+, \alpha(q)=\al}\left(4\pi A(q)+\log\left(h(q)\prod_{1\leq i\leq\ell,p_i\neq q}e^{-4\pi\alpha_i G_{p_i}(q)}\right)\right).
\end{align}

\begin{rem} If the reader is familiar with \cite{DJLW}, it is easy to ask that,
suppose $(u_k)$ blows up at some $p$ with $h(p)>0$ and $\alpha(p)=\al$, can we add condition on $h$ at $p$ like what were done by Ding-Jost-Li-Wang
to derive a sufficient condition for the existence of Eq. \eqref{eq-u}? In fact,
since of the conical singularities, the blow up (if happens) must at the most singular point $p$. This is the first thing that prevents us to using  the value of
$\Delta\log h(p)$ to given a sufficient condition; the second thing can be noticed by dedicate calculations like in \cite{DJLW}, which can be interpreted by
\begin{align*}
\int r^{-1}dr = \log r+c\to-\infty~~~\text{and}~~~\int r^{-1-2\al}dr = r^{-2\al}+c\to 0
\end{align*}
as $r\to0$ (for fixed $c$).
\end{rem}

\section{Complement of the proofs of Theorem \ref{thm-zhu-1} and \ref{thm-zhu-2}}
In this last section, we complete the proofs of our main theorems.

It is easy to see that Theorem \ref{thm-zhu-1} follows from \eqref{est_2} and \eqref{bound} directly. Notice that, either (i) or (ii) in Theorem \ref{thm-zhu-2} holds, we know form
Theorem \ref{thm-zhu-1} and its proof that $(u_k)$ does not blow up, then $J_{\rh}$ attains its infimum and Eq. (\ref{eq-u}) has a solution.  $\hfill{\square}$

\vspace{2cm}

\textbf{Acknowledgement} The main part of this paper was finished when the author visited School of Mathematics Science and China-France Mathematics Center at University of Science and Technology of China. He would like to thank them for their enthusiasm and the excellent working conditions they supplied for him.

\end{document}